\documentclass[]{interact}

\usepackage[numbers,sort&compress]{natbib}
\bibpunct[, ]{[}{]}{,}{n}{,}{,}
\renewcommand\bibfont{\fontsize{10}{12}\selectfont}
\makeatletter
\def\NAT@def@citea{\def\@citea{\NAT@separator}}
\makeatother

\usepackage{amsthm,amssymb,enumerate,mathtools,lmodern}

\usepackage{geometry}
\usepackage[colorlinks]{hyperref}

\theoremstyle{plain}
\newtheorem{theorem}{Theorem}[section]
\newtheorem{lemma}[theorem]{Lemma}
\newtheorem{corollary}[theorem]{Corollary}

\theoremstyle{definition}
\newtheorem{definition}[theorem]{Definition}

\theoremstyle{remark}

\newcommand{\N}{\mathbb{N}}
\newcommand{\R}{\mathbb{R}}
\newcommand{\bb}[1]{\boldsymbol{#1}}
\newcommand{\leqdef}{\vcentcolon=}
\newcommand{\rd}{{\rm d}}
\newcommand{\tmean}{\bar{t}}

\allowdisplaybreaks

\begin{document}

\articletype{ARTICLE}

\title{Complete monotonicity of a ratio of gamma functions and some combinatorial inequalities for multinomial coefficients}

\author{
\name{Fr\'ed\'eric Ouimet \thanks{Email: ouimetfr@caltech.edu}}
\affil{California Institute of Technology, Pasadena, 91125, USA.}
}

\maketitle

\begin{abstract}
    For $m,n\in \N$, let $0 < \alpha_i,\beta_j,\lambda_{ij} \leq 1$ be such that $\sum_{j=1}^n \lambda_{ij} = \alpha_i$, $\sum_{i=1}^m \lambda_{ij} = \beta_j$, and \\[-3.2mm]

    \noindent
    $\sum_{i=1}^m \alpha_i = \sum_{j=1}^n \beta_j \leq 1$.
    We prove that the ratio of gamma functions
    \vspace{-2mm}
    \begin{equation*}
        \hspace{-15mm}t \mapsto \frac{\prod_{i=1}^m \Gamma(\alpha_i t + 1) \prod_{j=1}^n \Gamma(\beta_j t + 1)}{\prod_{i=1}^m \prod_{j=1}^n \Gamma(\lambda_{ij} t + 1)}
    \end{equation*}
    is logarithmically completely monotonic on $(0,\infty)$.
    This result complements the logarithmically complete monotonicity of multinomial probabilities shown in
    [F.\ Ouimet (2018), Complete monotonicity of multinomial probabilities and its
    application to Bernstein estimators on the simplex, J.\ Math.\ Anal.\ Appl., 466(2), 1609-1617, \href{http://www.ams.org/mathscinet-getitem?mr=MR3825458}{MR3825458}],
    [F.\ Qi, D.-W.\ Niu, D.\ Lim, \& B.-N.\ Guo (2018), Some logarithmically completely monotonic functions and inequalities
	for multinomial coefficients and multivariate beta functions, Preprint, 1-13, \href{https://hal.archives-ouvertes.fr/hal-01769288}{hal-01769288}],
    and the recent survey of
    [F.\ Qi \& R.\ P.\ Argawal (2019), On complete monotonicity for several classes of functions
    related to ratios of gamma functions, J.\ Inequal.\ Appl., Paper No.\ 36, 42 pp, \href{http://www.ams.org/mathscinet-getitem?mr=MR3908972}{MR3908972}]
    on the complete monotonicity of functions related to ratios of gamma functions.
    As a consequence of the log-convexity, we obtain new combinatorial inequalities for multinomial coefficients.
    \phantom{\cite{MR3825458,hal-01769288,MR3908972}}
\end{abstract}


\vspace{-7mm}
\section{Introduction}

Completely monotonic functions on $(0,\infty)$ are non-negative functions for which derivatives of all orders exist on $(0,\infty)$ and alternate in sign (starting with the negative sign).
Typical examples are $t^{-1}$, $(1 + t)^{-1}$, $e^{-t}$, etc.
A famous theorem of \cite{MR1555269} shows that the set of completely monotonic functions $\phi : [0,\infty)\to \R$ such that $\phi(0) = 1$ coincides with the set of Laplace transforms, see e.g.\ Section XIII.4 of \cite{MR0270403} for a simpler proof.
For a classic introduction to the theory of Laplace transforms, we refer the reader to \cite{MR0005923}.
For a survey on the complete monotonicity of functions related to ratios of gamma functions, see \cite{MR3460546,MR3908972}.

\newpage
Below are the formal definitions of complete monotonicity and logarithmically complete monotonicity that we use.

\begin{definition}\label{def:complete.monotonicity}
    A function $t \mapsto g(t)$ is said to be completely monotonic on $(0,\infty)$ if $g$ has derivatives of all orders and satisfies
    \begin{equation}\label{eq:def:complete.monotonicity}
        (-1)^k g^{(k)}(t) \geq 0, \quad \text{for all } k\in \N_0, ~t\in (0,\infty).
    \end{equation}
\end{definition}

\begin{definition}\label{def:logarithmic.complete.monotonicity}
    A function $t \mapsto g(t)$ is said to be logarithmically completely monotonic on $(0,\infty)$ if $(-\log g)'$ is completely monotonic on $(0,\infty)$.
\end{definition}

It turns out that logarithmically completely monotonic functions are completely monotonic, see e.g.\ \cite[p.83]{MR0072370}.

\begin{lemma}\label{lem:Alzer.2018.lemma.2}
    Let $g : (0,\infty) \rightarrow (0,\infty)$.
    If $(-\log g)'$ is completely monotonic on $(0,\infty)$, then $g$ is completely monotonic on $(0,\infty)$.
\end{lemma}

\vspace{-7mm}
\subsection{Structure of the paper}

In the next section, we state and prove our main result (Theorem \ref{thm:completely.monotonic}), then we deduce new combinatorial inequalities in Corollary \ref{cor:combinatorial.inequality}.
In the appendix, the reader will find a technical inequality (and its proof) which is the key step in the proof of Theorem \ref{thm:completely.monotonic}.

\vspace{-4mm}
\section{Main result}\label{sec:main.result}

\vspace{-1mm}
\begin{theorem}\label{thm:completely.monotonic}
    For $m,n\in \N$, let $0 < \alpha_i,\beta_j,\lambda_{ij} \leq 1$ be such that
    \vspace{-3mm}
    \begin{equation*}
        \sum_{j=1}^n \lambda_{ij} = \alpha_i, \qquad \sum_{i=1}^m \lambda_{ij} = \beta_j,
    \end{equation*}
    
    \vspace{-4mm}
    \noindent
    and

    \vspace{-7.5mm}
    \begin{equation*}
        \sum_{i=1}^m \sum_{j=1}^n \lambda_{ij} = \sum_{i=1}^m \alpha_i = \sum_{j=1}^n \sum_{i=1}^m \lambda_{ij} = \sum_{j=1}^n \beta_j \leq 1.
    \end{equation*}
    Then, the function
    \vspace{-4mm}
    \begin{equation}\label{eq:thm:completely.monotonic}
        g(t) = \frac{\prod_{i=1}^m \Gamma(\alpha_i t + 1) \prod_{j=1}^n \Gamma(\beta_j t + 1)}{\prod_{i=1}^m \prod_{j=1}^n \Gamma(\lambda_{ij} t + 1)}
    \end{equation}

    \vspace{-1mm}
    \noindent
    is logarithmically completely monotonic on $(0,\infty)$, where $\Gamma$ denotes the classical Euler's gamma function, which is defined by $\Gamma(z) \leqdef \int_0^{\infty} u^{z-1} e^{-u} \rd u$ for $z > 0$.
    In particular, Lemma \ref{lem:Alzer.2018.lemma.2} implies that $g$ is completely monotonic on $(0,\infty)$.
\end{theorem}

\begin{proof}[Proof of Theorem \ref{thm:completely.monotonic}]
    Define $h(t) \leqdef -\log g(t)$.
    We have
    \vspace{-1mm}
    \begin{equation}\label{eq:h.prime}
        \begin{aligned}
            h'(t)
            &= - \sum_{i=1}^m \alpha_i \psi(\alpha_i t + 1) - \sum_{j=1}^n \beta_j \psi(\beta_j t + 1) + \sum_{i=1}^m \sum_{j=1}^n \lambda_{ij} \psi(\lambda_{ij} t + 1),
        \end{aligned}
    \end{equation}
    where $\psi \leqdef (\log \Gamma)' = \Gamma' / \Gamma$ is the digamma function.
    Using the integral representation
    \vspace{-1mm}
    \begin{equation}
        \psi'(z) = \int_0^{\infty} \frac{u e^{-(z - 1)u}}{e^u - 1} \rd u, \quad z\in (0,\infty),
    \end{equation}
    
    \newpage
    \noindent
    see \cite[p.260]{MR0167642}, we obtain
    \begin{align}
        h''(t)
        &= - \sum_{i=1}^m \alpha_i^2 \psi'(\alpha_i t + 1) - \sum_{j=1}^n \beta_j^2 \psi'(\beta_j t + 1) + \sum_{i=1}^m \sum_{j=1}^n \lambda_{ij}^2 \psi'(\lambda_{ij} t + 1) \notag \\[2mm]
        &= - \sum_{i=1}^m \int_0^{\infty} \frac{\alpha_i u \hspace{0.2mm} e^{-\alpha_i u \hspace{0.2mm} t}}{e^u - 1} \alpha_i \rd u - \sum_{j=1}^n \int_0^{\infty} \frac{\beta_j u \hspace{0.2mm} e^{-\beta_j u \hspace{0.2mm} t}}{e^u - 1} \beta_j \rd u + \sum_{i=1}^m \sum_{j=1}^n \int_0^{\infty} \frac{\lambda_{ij} u \hspace{0.2mm} e^{-\lambda_{ij} u \hspace{0.2mm} t}}{e^u - 1} \lambda_{ij} \rd u \notag \\[2mm]
        &= - \int_0^{\infty} se^{-s \hspace{0.2mm} t} \mathcal{J}_{(\lambda_{ij}/C)}(e^{s/C}) \rd s,
    \end{align}
    where $C \leqdef \sum_{i=1}^m \sum_{j=1}^n \lambda_{ij}$ and $\mathcal{J}_{(u_{ij})}(y)$ is defined in \eqref{eq:lem:Alzer.generalization}.
    By Lemma \ref{lem:Alzer.generalization}, for all $k\in \N$ and $t\in (0,\infty)$,
    \begin{equation}\label{eq:prop:completely.monotonic.almost.equation}
        (-1)^k h^{(k+1)}(t) = \int_0^{\infty} s^k e^{- s \hspace{0.2mm} t} \mathcal{J}_{(\lambda_{ij}/C)}(e^{s/C}) \rd s > 0.
    \end{equation}
    Since $h'$ is decreasing, we show that $\lim_{t\to\infty} h'(t) \geq 0$ to conclude the proof.

    If we apply the recurrence formula
    \begin{equation}
        \psi(z + 1) = \psi(z) + \frac{1}{z}, \quad z\in (0,\infty),
    \end{equation}
    see \cite[p.258]{MR0167642}, we obtain from \eqref{eq:h.prime} the representation
    \begin{equation}\label{eq:h.prime.decomposition}
        \begin{aligned}
            h'(t)
            &= \frac{- m - n + mn}{t} - \sum_{i=1}^m \alpha_i R(\alpha_i t) - \sum_{j=1}^n \beta_j R(\beta_j t) + \sum_{i=1}^m \sum_{j=1}^n \lambda_{ij} R(\lambda_{ij} t) \\
            &\quad - \sum_{i=1}^m \alpha_i \log(\alpha_i) - \sum_{j=1}^n \beta_j \log(\beta_j) + \sum_{i=1}^m \sum_{j=1}^n \lambda_{ij} \log(\lambda_{ij}),
        \end{aligned}
    \end{equation}
    where $R(z) \leqdef \psi(z) - \log z$.
    Using the asymptotic formula
    \begin{equation}
        R(z) = -\frac{1}{2z} - \frac{1}{12z^2} + O(z^{-4}), \quad \text{as } z\to \infty,
    \end{equation}
    see \cite[p.259]{MR0167642}, all the terms on the first line on the right-hand side of \eqref{eq:h.prime.decomposition} converge to $0$ as $t\to \infty$.
    Since $\sum_{j=1}^n \lambda_{ij} = \alpha_i > 0$ and $\sum_{j=1}^n \beta_j \leq 1$, Jensen's inequality applied to $- \log(\cdot)$ yields
    \vspace{-3mm}
    \begin{align}\label{eq:end}
        \lim_{t\to\infty} h'(t)
        &= - \sum_{i=1}^m \alpha_i \sum_{j=1}^n (\lambda_{ij} / \alpha_i) \log\bigg(\frac{\beta_j}{\lambda_{ij} / \alpha_i}\bigg) \notag \\[0.5mm]
        &\geq - \sum_{i=1}^m \alpha_i \log\Big(\sum_{j=1}^n \beta_j\Big) \notag \\[2.5mm]
        &\geq 0.
    \end{align}
    This ends the proof.
\end{proof}

\newpage
In the context of Theorem \ref{thm:completely.monotonic}, note that
\begin{equation}\label{eq:g.multinomial.coef.form}
    \begin{aligned}
        g(t)
        &= \frac{\prod_{i=1}^m \Gamma(\alpha_i t + 1) \prod_{j=1}^n \Gamma(\beta_j t + 1)}{\prod_{i=1}^m \prod_{j=1}^n \Gamma(\lambda_{ij} t + 1)} \\[1mm]
        &= \prod_{i=1}^m {\alpha_i t \choose \lambda_{i1} t, \lambda_{i2} t, \dots, \lambda_{in} t} \prod_{j=1}^n {\beta_j t \choose \lambda_{1j} t, \lambda_{2j} t, \dots, \lambda_{mj} t} \prod_{i=1}^m \prod_{j=1}^n \Gamma(\lambda_{ij} t + 1).
    \end{aligned}
\end{equation}
We are now ready to prove the new combinatorial inequalities.

\begin{corollary}\label{cor:combinatorial.inequality}
    Let $\nu\in \N$. For all $k\in \{1,2,\ldots,\nu\}$, choose $t_k\in (0,\infty)$ and let $\mu_k\in (0,1)$ be such that $\sum_{k=1}^\nu \mu_k = 1$.
    The following inequalities hold : \\[-4mm]
    \begin{enumerate}[\qquad(a)]
        \item $g(\sum_{k=1}^\nu \mu_k t_k) \leq \prod_{k=1}^\nu g(t_k)^{\mu_k}$, where equality holds if and only if all the $t_k$'s are the same. \\[-3.5mm]
        \item $\prod_{k=1}^\nu g(t_k) < g(\sum_{k=1}^\nu t_k)$. \\[-2mm]
        \item If $t_1 \leq t_3$, then $g(t_1 + t_2) g(t_3) \leq g(t_1) g(t_2 + t_3)$, where equality holds if and only if $t_1 = t_3$.
    \end{enumerate}
    Using \eqref{eq:g.multinomial.coef.form}, we can also write the above inequalities with multinomial coefficients :
    \begin{enumerate}[(a')]
        \item If we assume further that $\mu_k = 1/\nu$ for all $k$, and denote $\tmean \leqdef \frac{1}{\nu} \sum_{k=1}^{\nu} t_k$, then
        \begin{equation*}
            \hspace{-3mm}
            \begin{aligned}
                &\left[
                    \begin{gathered}
                        \prod_{i=1}^m {\alpha_i \tmean \choose \lambda_{i1} \tmean, \lambda_{i2} \tmean, \dots, \lambda_{in} \tmean} \prod_{j=1}^n {\beta_j \tmean \choose \lambda_{1j} \tmean, \lambda_{2j} \tmean, \dots, \lambda_{mj} \tmean} \\
                        \cdot \prod_{i=1}^m \prod_{j=1}^n {\lambda_{ij} \tmean \choose \lambda_{ij} \tfrac{t_1}{\nu}, \lambda_{ij} \tfrac{t_2}{\nu}, \dots, \lambda_{ij} \tfrac{t_{\nu}}{\nu}}
                    \end{gathered}
                \right] \\[1mm]
                &\leq \prod_{k=1}^{\nu} \left[
                    \begin{gathered}
                        \prod_{i=1}^m {\alpha_i t_k \choose \lambda_{i1} t_k, \lambda_{i2} t_k, \dots, \lambda_{in} t_k} \prod_{j=1}^n {\beta_j t_k \choose \lambda_{1j} t_k, \lambda_{2j} t_k, \dots, \lambda_{mj} t_k} \\
                        \cdot \prod_{i=1}^m \prod_{j=1}^n {\lambda_{ij} t_k \choose \lambda_{ij} \frac{t_k}{\nu}, \lambda_{ij} \frac{t_k}{\nu}, \dots, \lambda_{ij} \frac{t_k}{\nu}}
                    \end{gathered}
                    \right]^{1/\nu},
            \end{aligned}
        \end{equation*}
        where equality holds if and only if all the $t_k$'s are the same.
        %%%
        \item
        \begin{equation*}
            \begin{aligned}
                &\prod_{k=1}^{\nu} \left[\prod_{i=1}^m {\alpha_i t_k \choose \lambda_{i1} t_k, \lambda_{i2} t_k, \dots, \lambda_{in} t_k} \prod_{j=1}^n {\beta_j t_k \choose \lambda_{1j} t_k, \lambda_{2j} t_k, \dots, \lambda_{mj} t_k}\right] \\[1mm]
                &\leq \prod_{i=1}^m {\alpha_i \sum_{i=1}^{\nu} t_k \choose \lambda_{i1} \sum_{i=1}^{\nu} t_k, \lambda_{i2} \sum_{i=1}^{\nu} t_k, \dots, \lambda_{in} \sum_{i=1}^{\nu} t_k} \\[1mm]
                &\quad\cdot \prod_{j=1}^n {\beta_j \sum_{i=1}^{\nu} t_k \choose \lambda_{1j} \sum_{i=1}^{\nu} t_k, \lambda_{2j} \sum_{i=1}^{\nu} t_k, \dots, \lambda_{mj} \sum_{i=1}^{\nu} t_k} \cdot \prod_{i=1}^m \prod_{j=1}^n {\lambda_{ij} \sum_{k=1}^{\nu} t_k \choose \lambda_{ij} t_1, \lambda_{ij} t_2, \dots, \lambda_{ij} t_{\nu}}.
            \end{aligned}
        \end{equation*}
        %%%
        \item
        If $t_1 \leq t_3$, then
        \begin{equation*}
            \begin{aligned}
                &\prod_{i=1}^m {\alpha_i (t_1 + t_2) \choose \lambda_{i1} (t_1 + t_2), \lambda_{i2} (t_1 + t_2), \dots, \lambda_{in} (t_1 + t_2)} {\alpha_i t_3 \choose \lambda_{i1} t_3, \lambda_{i2} t_3, \dots, \lambda_{in} t_3} \\[1mm]
                &\quad\cdot \prod_{j=1}^n {\beta_j (t_1 + t_2) \choose \lambda_{1j} (t_1 + t_2), \lambda_{2j} (t_1 + t_2), \dots, \lambda_{mj} (t_1 + t_2)} {\beta_j t_3 \choose \lambda_{1j} t_3, \lambda_{2j} t_3, \dots, \lambda_{mj} t_3} \\[1mm]
                &\qquad\cdot \prod_{i=1}^m \prod_{j=1}^n {\lambda_{ij} (t_1 + t_2 + t_3) \choose \lambda_{ij} t_1, \lambda_{ij} (t_2 + t_3)} \\[1mm]
                &\leq \prod_{i=1}^m {\alpha_i t_1 \choose \lambda_{i1} t_1, \lambda_{i2} t_1, \dots, \lambda_{in} t_1} {\alpha_i (t_2 + t_3) \choose \lambda_{i1} (t_2 + t_3), \lambda_{i2} (t_2 + t_3), \dots, \lambda_{in} (t_2 + t_3)} \\[1mm]
                &\quad\cdot \prod_{j=1}^n {\beta_j t_1 \choose \lambda_{1j} t_1, \lambda_{2j} t_1, \dots, \lambda_{mj} t_1} {\beta_j (t_2 + t_3) \choose \lambda_{1j} (t_2 + t_3), \lambda_{2j} (t_2 + t_3), \dots, \lambda_{mj} (t_2 + t_3)} \\[1mm]
                &\qquad\cdot \prod_{i=1}^m \prod_{j=1}^n {\lambda_{ij} (t_1 + t_2 + t_3) \choose \lambda_{ij} (t_1 + t_2), \lambda_{ij} t_3},
            \end{aligned}
        \end{equation*}
        where equality holds if and only if $t_1 = t_3$.
    \end{enumerate}
\end{corollary}

\begin{proof}[Proof of Corollary \ref{cor:combinatorial.inequality}]
    By \eqref{eq:prop:completely.monotonic.almost.equation}, $g$ is strictly log-convex, which implies $(a)$ by definition. Point $(b)$ follows from Lemma 3 in \cite{MR3730425} because $g$ is differentiable on $[0,\infty)$, $g(0) = 1$ and $g$ is (strictly) positive, (strictly) decreasing and strictly log-convex on $(0,\infty)$. Point $(c)$ follows by adapting the proof of Corollary 3 in \cite{MR3730425} using \eqref{eq:prop:completely.monotonic.almost.equation}.
\end{proof}

\appendix
\section{A technical lemma}\label{sec:tech.lemma}

We needed the following key inequality in the proof of Theorem \ref{thm:completely.monotonic}.

\begin{lemma}\label{lem:Alzer.generalization}
    For $m,n\in \N$, let $0 < U_i, V_j, u_{ij} \leq 1$ be such that
    \begin{equation*}
        \sum_{j=1}^n u_{ij} = U_i, \qquad \sum_{i=1}^m u_{ij} = V_j,
    \end{equation*}
    and
    \begin{equation*}
        \sum_{i=1}^m \sum_{j=1}^n u_{ij} = \sum_{i=1}^m U_i = \sum_{j=1}^n \sum_{i=1}^m u_{ij} = \sum_{j=1}^n V_j = 1.
    \end{equation*}
    Then, for any given $y > 1$,
    \begin{equation}\label{eq:lem:Alzer.generalization}
        \mathcal{J}_{(u_{ij})}(y) \leqdef \sum_{i=1}^m \frac{1}{y^{1/U_i} - 1} + \sum_{j=1}^n \frac{1}{y^{1/V_j} - 1} - \sum_{i=1}^m \sum_{j=1}^n \frac{1}{y^{1/u_{ij}} - 1} > 0,
    \end{equation}
    where $(u_{ij})$ is a shorthand for the matrix $(u_{ij})_{1 \leq i \leq m; 1 \leq j \leq n}$.
\end{lemma}

\newpage
\begin{proof}
    First, we write $-\mathcal{J}_{(u_{ij})}(y)$ as a function of the variables $(u_{ij})_{1 \leq i \leq m-1; 1 \leq j \leq n-1}$ when viewing the $U_i$'s and $V_j$'s as fixed :
    \begin{align}\label{eq:lem:Alzer.generalization.rewrite}
        -\mathcal{J}_{(u_{ij})}(y)
        &= \sum_{i=1}^m \sum_{j=1}^n \frac{1}{y^{1/u_{ij}} - 1} \hspace{-20mm}\underbrace{- \, \sum_{i=1}^m \frac{1}{y^{1/U_i} - 1} - \, \sum_{j=1}^n \frac{1}{y^{1/V_j} - 1}}_{~~~~~~~~~~~~~~~~~\text{this is independent of $(u_{ij})_{1 \leq i \leq m-1; 1 \leq j \leq n-1}$; call it $C(y)$}} \notag \\
        &= \sum_{i=1}^{m-1} \sum_{j=1}^{n-1} \frac{1}{y^{1/u_{ij}} - 1} + \sum_{i=1}^{m-1} \frac{1}{y^{1/u_{in}} - 1} + \sum_{j=1}^{n-1} \frac{1}{y^{1/u_{mj}} - 1} + \frac{1}{y^{1/u_{mn}} - 1} + C(y) \notag \\[5mm]
        &= \sum_{i=1}^{m-1} \sum_{j=1}^{n-1} \frac{1}{y^{1/u_{ij}} - 1} + \sum_{i=1}^{m-1} \frac{1}{y^{1/(U_i - \sum_{j=1}^{n-1} u_{ij})} - 1} + \sum_{j=1}^{n-1} \frac{1}{y^{1/(V_j - \sum_{i=1}^{m-1} u_{ij})} - 1} \notag \\
        &\qquad+ \frac{1}{y^{1/(V_n - \sum_{i=1}^{m-1} U_i + \sum_{i=1}^{m-1} \sum_{j=1}^{n-1} u_{ij})} - 1} + C(y).
    \end{align}
    From the proof of Lemma 1 in \cite{MR3730425}, we know that $\frac{\partial^2}{\partial c^2} (y^{1/c} - 1)^{-1} > 0$ for all $c\in (0,1)$.
    For convenience, here are the computations (with $t = y^{1/c}$, and $y > 1$ by assumption) :
    \begin{align}\label{eq:lem:Alzer.generalization.Alzer.computation}
        \frac{c^2 (t - 1)^3}{2 t (\log t)^2} \frac{\partial^2}{\partial c^2} (y^{1/c} - 1)^{-1}
        = \frac{t + 1}{2} - \frac{t - 1}{\log t}
        = \frac{t + 1}{2 \log t} \int_1^t \frac{(s - 1)^2}{s (s + 1)^2} \rd s > 0.
    \end{align}
    Therefore, everywhere in the open set
    \begin{equation}
        \mathcal{O} = \Big\{(u_{ij})_{1 \leq i \leq m-1; 1 \leq j \leq n-1}\in (0,\infty)^{(m-1) \times (n-1)} : \sum_{i=1}^{m-1} u_{ij} < V_j, \sum_{j=1}^{n-1} u_{ij} < U_i\Big\},
    \end{equation}
    we have (for $1 \leq k \neq k' \leq m-1$ and $1 \leq \ell \neq \ell' \leq n-1$) :
    \begin{align}
        \frac{\partial^2}{\partial u_{k\ell}\partial{u_{k\ell}}} \big[-\mathcal{J}_{(u_{ij})}(y)\big]
        &= \frac{\partial^2}{\partial u_{k\ell}\partial{u_{k\ell}}} \frac{1}{y^{1/u_{k\ell}} - 1} \notag \\
        &\quad+ \frac{\partial^2}{\partial u_{k\ell}\partial{u_{k\ell}}} \frac{1}{y^{1/(U_k - \sum_{j=1}^{n-1} u_{kj})} - 1} \notag \\
        &\quad+ \frac{\partial^2}{\partial u_{k\ell}\partial{u_{k\ell}}} \frac{1}{y^{1/(V_{\ell} - \sum_{i=1}^{m-1} u_{i\ell})} - 1} \notag \\
        &\quad+ \frac{\partial^2}{\partial u_{k\ell}\partial{u_{k\ell}}} \frac{1}{y^{1/(V_n - \sum_{i=1}^{m-1} U_i + \sum_{i=1}^{m-1} \sum_{j=1}^{n-1} u_{ij})} - 1} \notag \\[1mm]
        &= a_{k\ell} + a_{kn} + a_{m\ell} + a_{mn}, \\[1mm]
        \frac{\partial^2}{\partial u_{k\ell}\partial{u_{k\ell'}}} \big[-\mathcal{J}_{(u_{ij})}(y)\big]
        &= \frac{\partial^2}{\partial u_{k\ell}\partial{u_{k\ell'}}} \frac{1}{y^{1/(U_k - \sum_{j=1}^{n-1} u_{kj})} - 1} \notag \\
        &\quad+ \frac{\partial^2}{\partial u_{k\ell}\partial{u_{k\ell'}}} \frac{1}{y^{1/(V_n - \sum_{i=1}^{m-1} U_i + \sum_{i=1}^{m-1} \sum_{j=1}^{n-1} u_{ij})} - 1} \notag \\[1mm]
        &= a_{kn} + a_{mn}, \\[1mm]
        \frac{\partial^2}{\partial u_{k\ell}\partial{u_{k'\ell}}} \big[-\mathcal{J}_{(u_{ij})}(y)\big]
        &= \frac{\partial^2}{\partial u_{k\ell}\partial{u_{k'\ell}}} \frac{1}{y^{1/(V_{\ell} - \sum_{i=1}^{m-1} u_{i\ell})} - 1} \notag \\
        &\quad+ \frac{\partial^2}{\partial u_{k\ell}\partial{u_{k'\ell}}} \frac{1}{y^{1/(V_n - \sum_{i=1}^{m-1} U_i + \sum_{i=1}^{m-1} \sum_{j=1}^{n-1} u_{ij})} - 1} \notag \\[1mm]
        &= a_{m\ell} + a_{mn}, \\[1mm]
        \frac{\partial^2}{\partial u_{k\ell}\partial{u_{k'\ell'}}} \big[-\mathcal{J}_{(u_{ij})}(y)\big]
        &= \frac{\partial^2}{\partial u_{k\ell}\partial{u_{k'\ell'}}} \frac{1}{y^{1/(V_n - \sum_{i=1}^{m-1} U_i + \sum_{i=1}^{m-1} \sum_{j=1}^{n-1} u_{ij})} - 1} \notag \\[1mm]
        &= a_{mn},
    \end{align}
    where $a_{ij} > 0$ for all $1 \leq i \leq m$, $1 \leq j \leq n$, on $\mathcal{O}$ by \eqref{eq:lem:Alzer.generalization.Alzer.computation}.
    In other words, the Hessian matrix of $-\mathcal{J}_{(u_{ij})}(y)$, as a function of the variables $(u_{ij})_{1 \leq i \leq m-1; 1 \leq j \leq n-1}$, is equal to
    \begingroup
    \setlength\arraycolsep{1.5pt}
    \begin{align}\label{eq:hessian.matrix.decomposition}
        &
        \begin{pmatrix}
            \boldsymbol{A}_1 &\boldsymbol{0} &\cdots &\boldsymbol{0} \\
            \boldsymbol{0} &\boldsymbol{A}_2 &\ddots &\vdots \\
            \vdots &\ddots &\ddots &\boldsymbol{0} \\[1mm]
            \boldsymbol{0} &\cdots &\boldsymbol{0} &\boldsymbol{A}_{m-1}
        \end{pmatrix}
        +
        \begin{pmatrix}
            \boldsymbol{B}_1 &\boldsymbol{0} &\cdots &\boldsymbol{0} \\
            \boldsymbol{0} &\boldsymbol{B}_2 &\ddots &\vdots \\
            \vdots &\ddots &\ddots &\boldsymbol{0} \\[1mm]
            \boldsymbol{0} &\cdots &\boldsymbol{0} &\boldsymbol{B}_{m-1}
        \end{pmatrix}
        +
        \begingroup
        \setlength\arraycolsep{4pt}
        \begin{pmatrix}
            \boldsymbol{C} &\boldsymbol{C} &\cdots &\boldsymbol{C} \\
            \boldsymbol{C} &\boldsymbol{C} &\ddots &\vdots \\
            \vdots &\ddots &\ddots &\boldsymbol{C} \\[1mm]
            \boldsymbol{C} &\cdots &\boldsymbol{C} &\boldsymbol{C}
        \end{pmatrix}
        \endgroup
        +
        a_{mn} \bb{1}_{(m-1)(n-1)} \notag \\[3mm]
        &\leqdef \mathrm{(I)} + \mathrm{(II)} + \mathrm{(III)} + \mathrm{(IV)}.
    \end{align}
    \endgroup
    where $\boldsymbol{A}_i = \text{diag}((a_{ij})_{1 \leq j \leq n-1})$, $\boldsymbol{B}_i = a_{in} \bb{1}_{(n-1)}$, $\boldsymbol{C} = \text{diag}((a_{mj})_{1 \leq j \leq n-1})$ and $\bb{1}_{\mu}$ denotes the $\mu\times \mu$ matrix of ones.
    Since all the $a_{ij}$'s are positive on $\mathcal{O}$, it is easy to verify that $\mathrm{(I)}$ is positive definite and $\mathrm{(II)}$, $\mathrm{(III)}$ and $\mathrm{(IV)}$ are positive semi-definite.
    Indeed, for any non-zero vector $\bb{x}\in \R^{(m-1)(n-1)}\backslash \{\bb{0}\}$, write it as the vertical concatenation of the column vectors $(\bb{x}_i)_{1 \leq i \leq m-1}$ where $\bb{x}_i\leqdef (x_{ij})_{1 \leq j \leq n-1}$, then
    \begin{align*}
        \bb{x}^{\top}\mathrm{(I)} \, \bb{x}
        &= \sum_{i=1}^{m-1} \bb{x}_i^{\top} \boldsymbol{A}_i \bb{x}_i = \sum_{i=1}^{m-1} \sum_{j=1}^{n-1} a_{ij} x_{ij}^2 > 0, \\[2mm]
        %%%
        \bb{x}^{\top}\mathrm{(II)} \, \bb{x}
        &= \sum_{i=1}^{m-1} \bb{x}_i^{\top} \boldsymbol{B}_i \bb{x}_i = \sum_{i=1}^{m-1} \sum_{j=1}^{n-1} \sum_{j'=1}^{n-1} a_{in} x_{ij}x_{ij'}
        = \sum_{i=1}^{m-1} a_{in} \Big(\sum_{j=1}^{n-1} x_{ij}\Big)^2 \geq 0, \\[2mm]
        %%%
        \bb{x}^{\top}\mathrm{(III)} \, \bb{x}
        &= \sum_{i=1}^{m-1} \sum_{i'=1}^{m-1} \bb{x}_i^{\top} \boldsymbol{C} \bb{x}_{i'} = \sum_{i=1}^{m-1} \sum_{i'=1}^{m-1} \sum_{j=1}^{n-1} a_{mj} x_{ij}x_{i'j}
        = \sum_{j=1}^{n-1} a_{mj} \Big(\sum_{i=1}^{m-1} x_{ij}\Big)^2 \geq 0, \\[2mm]
        %%%
        \bb{x}^{\top}\mathrm{(IV)} \, \bb{x} &= a_{mn} \Big(\sum_{i=1}^{m-1} \sum_{j=1}^{n-1} x_{ij}\Big)^2 \geq 0.
    \end{align*}
    By linearity, this means that the Hessian matrix of $-\mathcal{J}_{(u_{ij})}(y)$ is positive definite.
    Since the second-order partial derivatives are continuous on the open and convex set $\mathcal{O}$, it implies that $\mathcal{J}_{(u_{ij})}(y)$, as a function of the variables $(u_{ij})_{1 \leq i \leq m-1; 1 \leq j \leq n-1}$, is strictly concave on $\mathcal{O}$.
    A strictly concave function on a convex set minimizes at the extremal points of its closure.
    Here, these are the points $(u_{ij})_{1 \leq i \leq m-1; 1 \leq j \leq n-1}$ such that $u_{i^{\star}j^{\star}} = 1$ for some $1 \leq i^{\star} \leq m-1$ and $1 \leq j^{\star} \leq n-1$, and such that $u_{ij} = 0$ for all other $i \neq i^{\star}$ and $j \neq j^{\star}$.
    It is easy to verify that $\mathcal{J}_{(u_{ij})}(y) = (y - 1)^{-1} > 0$ in \eqref{eq:lem:Alzer.generalization} for any such point.
    Hence, $\mathcal{J}_{(u_{ij})}(y) > 0$ on $\mathcal{O}$, which was our claim.
\end{proof}

\section*{Acknowledgement}

We thank the referee for his careful review of the manuscript, and we thank P.\ Da Silva for helping us closing the argument in the proof of Lemma \ref{lem:Alzer.generalization}.

\section*{Disclosure statement}

I have no conflict of interest to disclose.

\section*{Funding}

F. Ouimet is supported by a postdoctoral fellowship from the NSERC (PDF) and a postdoctoral fellowship supplement from the FRQNT (B3X).

%
% ----------  B I B L I O G R A P H Y  ----------
%

%\section*{References}

\def\bibindent{3em}

\bibliographystyle{authordate1}
\bibliography{Ouimet_2019_ratio_gammas_bib}

\end{document}